\documentclass[a4paper, pdftex]{amsart}

\usepackage[utf8]{inputenc}
\usepackage[T1]{fontenc}
\usepackage{amsmath}
\usepackage{amsfonts}
\usepackage{amsthm}
\usepackage{amssymb}
\usepackage{thmtools}
\usepackage{mathtools}
\usepackage{hyperref}
\usepackage[capitalize,nameinlink,noabbrev]{cleveref}
\usepackage[all, cmtip]{xy}
\usepackage[textsize = tiny]{todonotes}
\usepackage{makecell}

\usepackage{mathabx} 

\usepackage[backend=bibtex, style=numeric, url=true, hyperref, maxbibnames=99, sorting=nyt, eprint=false, isbn=false, sortcites=true]{biblatex}
\AtEveryBibitem{\clearlist{language}}
\begin{document}
\newtheorem*{thm*}{Theorem}
\newtheorem{theorem}{Theorem}[section]
\newtheorem{corollary}[theorem]{Corollary}
\newtheorem{lemma}[theorem]{Lemma}
\newtheorem{fact}[theorem]{Fact}
\newtheorem*{fact*}{Fact}
\newtheorem{proposition}[theorem]{Proposition}
\newtheorem{claim}{Claim}

\newcounter{theoremalph}
\renewcommand{\thetheoremalph}{\Alph{theoremalph}}
\newtheorem{thmAlph}[theoremalph]{Theorem}
\theoremstyle{definition}
\newtheorem{question}[theorem]{Question}
\newtheorem{definition}[theorem]{Definition}

\theoremstyle{remark}
\newtheorem{remark}[theorem]{Remark}
\newtheorem{observation}[theorem]{Observation}
\newtheorem{example}[theorem]{Example}


\renewcommand{\ll}{\left\langle}
\newcommand{\rr}{\right\rangle}
\newcommand{\ls}{\left\{}
\newcommand{\rs}{\right\}}
\newcommand{\sm}{\setminus}

\newcommand{\mcA}{\ensuremath{\mathcal{A}}}
\newcommand{\mcB}{\ensuremath{\mathcal{B}}}
\newcommand{\mcC}{\ensuremath{\mathcal{C}}}
\newcommand{\mcE}{\ensuremath{\mathcal{E}}}
\newcommand{\mcF}{\ensuremath{\mathcal{F}}}
\newcommand{\mcG}{\ensuremath{\mathcal{G}}}
\newcommand{\mcH}{\ensuremath{\mathcal{H}}}
\newcommand{\mcI}{\ensuremath{\mathcal{I}}}
\newcommand{\mcJ}{\ensuremath{\mathcal{J}}}
\newcommand{\mcM}{\ensuremath{\mathcal{M}}}
\newcommand{\mcN}{\ensuremath{\mathcal{N}}}
\newcommand{\mcO}{\ensuremath{\mathcal{O}}}
\newcommand{\mcP}{\ensuremath{\mathcal{P}}}
\newcommand{\mcQ}{\ensuremath{\mathcal{Q}}}
\newcommand{\mcR}{\ensuremath{\mathcal{R}}}
\newcommand{\mcS}{\ensuremath{\mathcal{S}}}
\newcommand{\mcU}{\ensuremath{\mathcal{U}}}
\newcommand{\mcV}{\ensuremath{\mathcal{V}}}
\newcommand{\mcW}{\ensuremath{\mathcal{W}}}
\newcommand{\mcZ}{\ensuremath{\mathcal{Z}}}
\newcommand{\mbK}{\ensuremath{\mathbb{K}}}
\newcommand{\mbQ}{\ensuremath{\mathbb{Q}}}
\newcommand{\mbR}{\ensuremath{\mathbb{R}}}
\newcommand{\mbZ}{\ensuremath{\mathbb{Z}}}
\newcommand{\mbN}{\ensuremath{\mathbb{N}}}

\newcommand{\on}[1]{\operatorname{#1}}
\newcommand{\set}[1]{\ensuremath{ \left\lbrace #1 \right\rbrace}}	
\newcommand{\grep}[2]{\ensuremath{\left\langle #1 \, \middle| \, #2\right\rangle}}
\newcommand{\real}[1]{\ensuremath{\left\lVert #1\right\rVert}}
\newcommand{\overbar}[1]{\mkern 2mu\overline{\mkern-2mu#1\mkern-2mu}\mkern 2mu}
\newcommand{\til}[1]{\widetilde{#1}}
\newcommand{\gen}[1]{\ensuremath{\left\langle #1 \right\rangle}}

\newcommand{\Aut}{\ensuremath{\operatorname{Aut}}}
\newcommand{\AutO}{\ensuremath{\operatorname{Aut}^0}}
\newcommand{\Out}{\ensuremath{\operatorname{Out}}}
\newcommand{\Outo}[1][A_\Gamma]{\ensuremath{\operatorname{Out}^0(#1)}}
\newcommand{\Stab}{\operatorname{Stab}}
\newcommand{\Sym}{\operatorname{Sym}}
\newcommand{\St}{\operatorname{St}}
\newcommand{\Stsymp}{\operatorname{St}^\omega}

\newcommand{\op}{\operatorname{op}}
\newcommand{\st}{\operatorname{st}}
\newcommand{\lk}{\operatorname{lk}}
\newcommand{\im}{\operatorname{im}}
\newcommand{\rk}{\operatorname{rk}}
\newcommand{\corank}{\operatorname{crk}}
\newcommand{\spn}{\operatorname{span}}
\newcommand{\Ind}{\operatorname{Ind}}
\newcommand{\vcd}{\operatorname{vcd}}
\newcommand{\id}{\operatorname{id}}
\newcommand{\len}{\operatorname{len}}

\newcommand{\GL}[2]{\ensuremath{\operatorname{GL}_{#1}(#2)}}
\newcommand{\SL}[2]{\ensuremath{\operatorname{SL}_{#1}(#2)}}
\newcommand{\Sp}[2]{\ensuremath{\operatorname{Sp}_{#1}(#2)}}
\newcommand{\SO}{\ensuremath{\operatorname{SO}}}


\newcommand{\FreeF}{\mathcal{F}}
\newcommand{\FreeS}{\mathcal{FS}}
\newcommand{\FreeSone}{\mathcal{FS}^1}
\newcommand{\bFreeS}{b\mathcal{FS}}
\newcommand{\bFreeSone}{b\mathcal{FS}^{1}}
\newcommand{\CyclicS}{\mathcal{FZ}}

\newcommand{\IA}{\mathrm{IA}} 
\newcommand{\FFS}{\mathcal{FF}} 
\newcommand{\Frames}{\mathcal{PB}} 
\newcommand{\PartialBases}{\mathcal{B}} 
\newcommand{\MaxCore}[1]{\mathring{#1}} 
\newcommand{\NonTrees}{\ensuremath{\operatorname{X}}} 
\newcommand{\Core}{\ensuremath{\operatorname{C}}} 
\newcommand{\FC}{\mathrm{FC}}
\newcommand{\F}{\mathcal{F}}
    
\newcommand{\FreeSred}[1][n]{\mathcal{FS}_{#1}^{r}}
\newcommand{\bFreeSred}{\mathcal{FS}^{r,*}}

\newcommand{\X}{\ensuremath{\operatorname{X}}}
\newcommand{\Sub}{\ensuremath{\operatorname{Sub}}}

 \newcommand{\tq}{\mathrel{{\ensuremath{\: : \: }}}}

\title[Partial basis complexes of freely decomposable groups]{Connectivity of partial basis complexes of freely decomposable groups}

\author{Benjamin Br{\"u}ck}
\address{Universit{\"a}t M{\"u}nster \\
Institut f{\"u}r Mathematische Logik und Grundlagenforschung \\
48149 M{\"u}nster, Germany}
\email{benjamin.brueck@uni-muenster.de}

\author{Kevin I. Piterman}
\address{Vrije Universiteit Brussel \\
Department of Mathematics and Data Science \\
1050 Brussels, Belgium}
\email{kevin.piterman@vub.be}



\begin{abstract}
We show that the complex of partial bases of the free group of rank $n$, where vertices are seen up to conjugation, is Cohen--Macaulay of dimension $n-1$.
This confirms a conjecture by Day and Putman.
We prove our results in the more general context of freely decomposable groups.
\end{abstract}

\maketitle

\section{Introduction}

\subsection{Motivation}
Showing that a complex that a group acts on is highly connected (i.e.~has vanishing homotopy groups in low degrees) often allows one to deduce algebraic and cohomological properties about the corresponding group.
In the present article, the acting group will mostly be the automorphism group of a free group or free product. We begin, however, by discussing the analogous situation for $\SL{n}{\mbZ}$, which is already better understood.
In \cref{sec:related_work} we provide a brief overview of further related settings.
For now, we will avoid technical definitions and focus on the bigger picture.

Recall that the symmetric space $\SO(n)\backslash \SL{n}{\mbR}$ admits a natural action of the group $\SL{n}{\mbZ}$.
In \cite{BS:Cornersarithmeticgroups}, Borel--Serre constructed a bordification of this symmetric space and showed that its boundary has a nice ``algebraic description''. Namely, it has the homotopy type of the Tits building associated with the group $\SL{n}{\mbQ}$.
This is a simplicial complex that is highly connected, in the sense that it has dimension $n-2$ and is homotopy equivalent to a wedge of $(n-2)$-spheres.
Its only non-trivial reduced homology group hence sits in degree $n-2$, and is called the Steinberg module.
As the Borel--Serre bordification is a contractible manifold with corners, Poincar\'e--Lefschetz duality relates the ordinary cohomology of $\SL{n}{\mbZ}$ to its cohomology with coefficients in the Steinberg module \cite[Chapter 13.1.4]{AB:Buildings}. 
In \cite{Lee1976}, Lee--Szczarba described an explicit generating set for this module.
More recently, Church--Farb--Putman \cite{Church2019} 
provided a different proof for the construction of such a generating set by using that another ``algebraically constructed'' complex, namely the complex of partial bases of $\mathbb Z^n$, is highly connected.
Then, duality allowed them to deduce that the rational cohomology of $\SL{n}{\mbZ}$ vanishes at its virtual cohomological dimension.
To summarise: the group $\SL{n}{\mbZ}$ acts on a ``geometric'' contractible space, namely the symmetric space. This space has a boundary with an ``algebraic'' description, the Tits building, which is highly connected. There is also another ``algebraic'' complex, the complex of partial bases of $\mathbb Z^n$, whose high connectivity gives a generating set for the homology of the algebraic boundary. Combining these results allows one to compute part of the cohomology of $\SL{n}{\mbZ}$.

In the setting of $\Out(F_n)$, the outer automorphism group of the free group $F_n$ of rank $n$, the situation is far less well understood.
Nevertheless, there is a commonly used analogue of the symmetric space, namely, the Culler--Vogtmann Outer space \cite{CV:Moduligraphsautomorphisms}. Its simplicial completion, the free splitting complex, can be seen as a bordification of it, which is contractible by \cite{Hat:Homologicalstabilityautomorphism}.
The free splitting complex has several ``geometric'' descriptions, via marked graphs (of groups) \cite{CV:Moduligraphsautomorphisms}, actions on trees \cite{GL:outerspacefree}, and sphere systems in a double handlebody \cite{Hat:Homologicalstabilityautomorphism}.
There is also an ``algebraic'' description of its boundary: in \cite{BG:Homotopytypecomplex},  Brück--Gupta showed that such a boundary has the homotopy type of the complex of free factor systems, defined by Handel--Mosher \cite{HandelMosher}. 
They also gave a lower bound on the connectivity of this complex \cite[Theorems B and D]{BG:Homotopytypecomplex} and showed that an important subcomplex, the free factor complex, is homotopy equivalent to a wedge of $(n-2)$-spheres\footnote{An $\Aut(F_n)$-version of the free factor complex was defined by Hatcher--Vogtmann \cite{HVFreeFactors}.
It is also known to be $(n-2)$-spherical \cite{HVFreeFactors}, but its homology does not give the dualising module of $\Aut(F_n)$, as was shown by Himes--Miller--Nariman--Putman \cite{HMNPduality}}.
However, this is a proper subcomplex of the complex of free factor systems, and it is still unknown whether the entire boundary is highly connected, i.e. it has the homotopy type of a wedge of spheres of the maximal possible dimension (see \cite[Question 1.1]{BG:Homotopytypecomplex} and \cite[Question on p.~9]{Vogtmann2024}).\footnote{Note that the relation between the homology of this boundary complex and the analogue of the Steinberg module is less straightforward in this setting, since Outer space and the free splitting complex are not manifolds. For recent progress in this direction, see the work of Wade--Wasserman \cite{Wade2024a}.}

The present article makes a step towards an answer to this question by showing that the complex of free factor systems has many other highly connected subcomplexes (see \cref{thm_connectivity_free_factor_systems} and \cref{thm_CM_frames_intro}).
These results are not sufficient to determine its entire homotopy type, but they let us prove a conjecture by Day--Putman concerning an $\Out(F_n)$-analogue of the complex of partial bases.

\subsection{A conjecture by Day--Putman}

Write $[x]$ for the conjugacy class of an element or subgroup $x$ of $F_n$.
The \emph{complex of partial bases of $F_n$} is the complex $\PartialBases_n$ whose maximal simplices are sets $\{ C_1, \ldots, C_n\}$ of conjugacy classes such that there exists a basis $\{v_1,\ldots,v_n\}$ of $F_n$ with $C_i=[v_i]$.
It is not hard to show that $\PartialBases_n$ is a pure simplicial complex of dimension $n-1$.
This complex was studied by Day and Putman in \cite{DP:complexpartialbases} when looking for generating sets of certain groups of automorphisms of $F_n$.
Concretely, they provided a topological proof of a theorem by Magnus that describes a generating set for $\IA_n$, the Torelli subgroup of $\Aut(F_n)$. This group is defined as the kernel of the map $\Aut(F_n)\to \GL{n}{\mathbb{Z}}$. To obtain their result, Day--Putman showed that $\PartialBases_n$ is connected for $n\geq 2$, simply connected if $n\geq 3$, and that the orbit complex $\PartialBases_n/\IA_n$ is $(n-2)$-connected, i.e.~spherical of dimension $n-1$. 
Then they conjecture \cite[Conjecture 1.1]{DP:complexpartialbases} that $\PartialBases_n$ is $(n-2)$-connected as well. 
In this text, we prove their conjecture:

\begin{theorem}
\label{thm_connectivity_frames}
The simplicial complex $\PartialBases_n$ is Cohen--Macaulay of dimension $n-1$.
In particular, $\PartialBases_n$ is $(n-2)$-connected.
\end{theorem}

\subsection{Free factor systems}
To prove the above result, we build on and extend work by the first-named author and Gupta \cite{BG:Homotopytypecomplex} and generalisations in \cite{Bru:buildingsfreefactor}, which relate $\PartialBases_n$ to the simplicial boundary of Culler--Vogtmann Outer space \cite{CV:Moduligraphsautomorphisms}.
In the present text, we will mostly refer to \cite{Bru:buildingsfreefactora}, which contains the material of both \cite{Bru:buildingsfreefactor} and \cite{BG:Homotopytypecomplex}.
In fact, we prove a more general version of \cref{thm_connectivity_frames} that not only applies to free groups but also to free products. To state it, we need to introduce some notation.

Let $A$ be a finitely generated group that is freely decomposable, i.e.~is isomorphic to a non-trivial free product. A \emph{free factor system} of $A$ is a finite set $\mcA = \ls [A_1], \ldots, [A_k] \rs$ of conjugacy classes of non-trivial proper subgroups $A_i\leq A$ such that $A = H \ast A_1 \ast \cdots \ast A_k$, where $H$ is some free subgroup of $A$ and $\ast$ denotes the free product.
We allow $\mcA$ to be empty, but note that the empty set is only a free factor system if $A$ is free itself.
The \emph{corank} of $\mcA$ is the rank of a free complement $H$, that is, $\corank(\mcA) = \rk(H)$, which is well-defined by the Grushko Theorem.

Let $\FFS(A)$ be the poset of free factor systems of $A$ that are proper (i.e. different from $\{[A]\}$), with ordering $\sqsubseteq$ given by refinement: We have $\mcA \sqsubseteq \mcA'$ if for all $C\in \mcA$ there exists $C'\in \mcA'$ such that a representative of $C$ is contained in some representative of $C'$.
For  $\mcA \in \FFS(A)$, we denote by  $\FFS(A,\mcA)$ the upper interval $\FFS(A)_{\sqsupset\mcA}$ (see \cite[Definition 4.2]{Bru:buildingsfreefactora}).
This is the poset of \textit{free factor systems relative to $\mcA$}, which was first studied by Handel--Mosher \cite{HandelMosher}.
By \cite[Proposition 6.3]{HandelMosher}, its dimension is $2\corank(\mcA)+|\mcA|-3$.
We show that its homotopy groups vanish up to codimension $\corank(\mcA)+1$:

\begin{theorem}
\label{thm_connectivity_free_factor_systems}
The poset $\FFS(A,\mcA)$ is $c$-connected for
\begin{equation*}
c = \max \lbrace \ \corank(\mcA)+|\mcA|-4,\ \corank(\mcA)-2 \ \rbrace.
\end{equation*}
\end{theorem}
For $A= F_n$ and $\mcA = \emptyset$, we have $c = n-2$ and the result is \cite[Theorem B]{BG:Homotopytypecomplex}.
In \cite[Proposition 4.30]{Bru:buildingsfreefactora}, it was shown that $\FFS(A,\mcA)$ is homotopy equivalent to the simplicial boundary of relative versions of Culler--Vogtmann Outer space \cite{CV:Moduligraphsautomorphisms} as defined by Guirardel--Levitt \cite{GL:outerspacefree}, so \cref{thm_connectivity_free_factor_systems} implies that these boundaries are highly connected.

Note that \cref{thm_connectivity_free_factor_systems} does not completely determine the homotopy type of $\FFS(A,\mcA)$ as
the latter has dimension $2\corank(\mcA)+|\mcA|-3$, which is greater than $c+1$ in general (see also \cite[Question 4.50]{Bru:buildingsfreefactora} and the comments around it). 
We hope that the strategy presented here and in \cite{BG:Homotopytypecomplex} can be further developed to shed more light on the homotopy type of $\FFS(A,\mcA)$.

\subsection{Frame complexes}
To obtain \cref{thm_connectivity_frames}, we consider a certain subposet of $\FFS(A, \mcA)$:
The \textit{frame complex} associated with $(A, \mcA)$, denoted by $\Frames(A, \mcA)$, is the subposet of $\FFS(A, \mcA) \cup \{ [A]\}$ given by free factor systems $\mcA'$ such that for each $[A']\in \mcA'$, we either have $[A']\in \mcA$ or $A'\cong\mbZ$.
Note that $\Frames(A, \mcA)$ is a subposet of $\FFS(A, \mcA)$ except if $\mcA = \emptyset$ and $A\cong \mathbb Z$, where $\Frames(A,\mcA) = \{ [A] \}$. 
In general, the elements of $\Frames(A, \mcA)$ are of the form $\ls [\langle v_1 \rangle ],\ldots, [\langle v_i \rangle]\rs \sqcup \mcA$, where $\{v_1,\ldots,v_i\}$ is a partial basis of the free group $H$ in a decomposition $A = H \ast A_1 \ast \cdots \ast A_k$.
The poset relation of $\FFS(A, \mcA)$ restricts to the subset relation on $\Frames(A, \mcA)$, so we can see $\Frames(A, \mcA)$ as a simplicial complex with $0$-simplices of the form $\ls [\langle v \rangle ]\rs \sqcup \mcA$. Since $\mcA$ is present in any simplex, we could remove it from all of them and get an abstract simplicial complex; however, for consistency, we will keep $\mcA$ in our simplices and rather think of $\Frames(A,\mcA)$ as a poset.

\begin{theorem}
\label{thm_CM_frames_intro}
The complex $\Frames(A,\mcA)$ is Cohen--Macaulay of dimension $\corank(\mcA)-1$.
\end{theorem}

\subsubsection*{Relation between the theorems}
We deduce \cref{thm_CM_frames_intro} from \cref{thm_connectivity_free_factor_systems} by showing that the inclusion map $\Frames(A,\mcA) \to \FFS(A,\mcA)$ is highly connected (see Section \ref{sec_poset_notation} for definitions).
For $A = F_n$ and $\mcA = \emptyset$, we have $\corank(\mcA)=n$ and $\Frames(F_n,\emptyset)$ is an inflation 
of the complex $\PartialBases_n$
(in the sense of \cref{def:inflation}),
so \cref{thm_connectivity_frames} is an immediate consequence of \cref{thm_CM_frames_intro}.

\subsection{Related work}
\label{sec:related_work}
The complex $\PartialBases_n$ comes with a natural action of $\Out(F_n)$, the outer automorphism group of the free group. More generally, relative outer automorphism groups $\Out(A,\mcA^t)$ (see \cite[Section 4.1.2]{Bru:buildingsfreefactora}) act on $\Frames(A, \mcA)$ and $\FFS(A,\mcA)$.
We would like to mention that variants of the complexes $\PartialBases_n$ and $\Frames(A, \mcA)$ related to several other groups are known to be highly connected as well:

The $\Aut(F_n)$-version of \cref{thm_connectivity_frames} was established by Sadofschi Costa \cite{SC:complexOfPartialBases}, who showed that the simplicial complex whose simplices are the partial bases of $F_n$ (without modding out by the conjugation action on vertices) is Cohen--Macaulay of dimension $n-1$.
In that case, the spanning map yields a poset map from the face poset of the partial basis complex to the free factor poset, which the author leverages via a fibre-type argument relying on the Cohen–Macaulayness of the free factor poset, established by Hatcher--Vogtmann \cite{HVFreeFactors}. In our setting, where we work with conjugacy classes, there is no such spanning map. This introduces an additional challenge, which we address by considering product-type configurations and graph posets instead. In particular, our approach differs in technique and is independent of \cite{SC:complexOfPartialBases}.

Furthermore, Maazen \cite{Maa:Homologystabilitygeneral} showed that a $\on{GL}_n$-version of $\PartialBases_n$ is Cohen--Macaulay;
Harer \cite{Har:Stabilityhomologymapping} proved that a version related to the mapping class group of a surface is highly connected (cf.~the introduction of \cite{DP:complexpartialbases});
Putman studied a version related to the integral symplectic group $\Sp{2n}{\mbZ}$ \cite{Brueck2023,Put:infinitepresentationTorelli}.
For an overview of how the connectivity of these complexes can be used to compute high-degree cohomology of the corresponding groups, see \cite{Brueck2024a}.

\subsection*{Acknowledgements}
We thank the anonymous referees for comments that helped to improve the exposition of this article.
The first author was supported by the Deutsche Forschungsgemeinschaft (Project 427320536–SFB 1442) and Germany’s Excellence Strategy grant EXC 2044–390685587.
The second author was supported by the Alexander von Humboldt Stiftung, and the FWO grant 12K1223N.

\section{Posets of graphs}
\label{sec_posets_of_graphs}

For the proof of \cref{thm_connectivity_free_factor_systems} in \cref{sec_proof_frames}, we compare $\FFS(A,\mcA)$ with a poset of free splittings of $A$, and study connectivity of fibres. These fibres coincide with certain finite posets of graphs, which we show to be highly connected in the present section.

\subsection{Posets and wedges of spheres}
\label{sec_poset_notation}

We start out by recalling some topological concepts.

We consider a point to be a (trivial) wedge of $d$-spheres for all $d\in \mbZ$, and for $d<0$, we consider a non-contractible wedge of $d$-spheres to be the same as the empty set.

Let $m$ be an integer.
By convention, every topological space is $m$-connected if $m\leq -2$.
A topological space $X$ is $(-1)$-connected if it is non-empty, and for $m\geq 0$, $X$ is $m$-connected if the homotopy groups of $X$ of degree $\leq m$ vanish regardless of the basepoint.
A continuous map $f:X\to Y$ is said to be $m$-connected if it induces an isomorphism between homotopy groups of degree $<m$, and an epimorphism in degree $m$ (again, regardless of the basepoint).
If $X,Y$ are finite-dimensional CW-complexes, by a slight abuse of language, we will say that a continuous map $f:X\to Y$ is highly connected if $f$ is an $n$-equivalence, where $n$ is either the dimension of $X$ or of $Y$.
When we need to specify the concrete connectivity of the map, we will indicate the grade of the equivalence.

An $n$-dimensional simplicial complex $K$ is spherical if it is $(n-1)$-connected. Equivalently, $K$ has the homotopy type of a wedge of $n$-spheres.
We say that $K$ is Cohen--Macaulay if for every simplex $\sigma\in K$, its link $\lk_K(\sigma) = \{\tau\in K\tq \tau\cup \sigma\in K, \tau\cap\sigma = \emptyset\}$ is spherical of dimension $\dim K - \dim\sigma - 1$ (and in particular $K$ is spherical by taking $\sigma$ the empty simplex).

If $X$ is a poset, we study its topological properties via its order complex.
We call a map $f\colon X \to Y$ between two posets a \emph{poset map} or an \textit{order-preserving map} if $x\leq y$ implies $f(x)\leq f(y)$.
Any order-preserving map between posets induces a simplicial map, and hence a continuous map, between their order complexes.
For two order-preserving maps $f,g:X\to Y$ such that $f\leq g$ (that is, $f(x)\leq g(x)$ for all $x\in X$), the induced maps between the geometric realisations are homotopy equivalent.

If $Y$ is a subposet of $X$, we call $f\colon X \to Y$ \emph{monotone} if either for all $x\in X$ we have $f(x)\leq x$, or for all $x\in X$ we have $f(x)\geq x$. 
By \cite[1.3]{Qui:Homotopypropertiesposet}, a monotone poset map $f:X \to X$ defines a homotopy equivalence from $X$ to its image $f(X)$.
In particular, a poset with a unique maximal or minimal element is contractible.

We say that $X$ is Cohen--Macaulay if its order complex is.
In this case, to check for the Cohen--Macaulay property, one only needs to verify that $X$ is spherical, and that every interval of $X$ is spherical of the right dimension (see, for instance, \cite{Qui:Homotopypropertiesposet}).

\subsection{Graphs}
In what follows, by a graph we mean a $1$-dimensional compact CW-complex with no isolated vertices.
For a graph $G$, we denote by $V(G)$ its set of vertices, and by $E(G)$ its set of edges.
Note that multiple edges and loops are allowed.
By a subgraph we mean an edge-induced subgraph.
That is, subgraphs of $G$ are in one-to-one correspondence with subsets of $E(G)$.
Recall that the valence of a vertex $v \in V(G)$ is the number of connected components obtained by deleting $v$ from a small enough neighbourhood of $v$, i.e.~the number of half-edges adjacent to $v$.
An edge that has an endpoint with valence one is called a \emph{leaf}.
An edge is said to be \emph{separating} if removing any point of its interior disconnects the graph.
In particular, we consider leaves to be separating.
The \emph{rank} of a connected graph $G$, denoted by $\rk(G)$, is the rank of its fundamental group.

A \textit{labelled graph} is a pair $(G,l)$ where $G$ is a graph and $\emptyset\subseteq l \subseteq V(G)$ is a subset of its vertex set.
We call $l$ the set of \emph{labelled vertices} of $G$ and write $k = k(l) = |l|$ for the number of labelled vertices.\footnote{We deviate from the notation in \cite{Bru:buildingsfreefactora} here, where the labelling is a map $l:\ls 1, \ldots, k \rs \to V(G)$. The reason is that we only care about the set of labelled vertices, i.e.~the image of such a map $l$ in the present work.}
A \textit{core subgraph} of a labelled graph $(G,l)$ is a non-empty subgraph such that every vertex of valence one is labelled.
Observe that every labelled graph $(H,l)$ contains a unique maximal core subgraph $\MaxCore{H}$: this is the subgraph obtained by repeatedly removing leaves with non-labelled endpoints.

Let $e\in E(G)$ be an edge of a labelled graph $(G,l)$.
We write $G\setminus e$ for the subgraph of $G$ on $E(G)\setminus \{e\}$ and $G/e$ for the graph obtained from $G$ by collapsing $e$.
The graphs $G\setminus e$ and $G/e$ have obvious labellings induced by $l$: the labelling of $G\setminus e$ is given by $l \cap V(G\setminus e)$ and the one of $G/e$ is obtained as the image of $l$ under the quotient map $V(G)\to V(G/e)$.
For example, if $e$ is a leaf with a labelled endpoint of valence one and an unlabelled endpoint of higher valence, then $G/e$ and $G\sm e$ are homeomorphic, but $G/e$ has one more labelled vertex than $G\sm e$ (namely the one coming from the higher valence endpoint of $e$ in $G$).
See \cref{fig:collapseVsQuotient}.

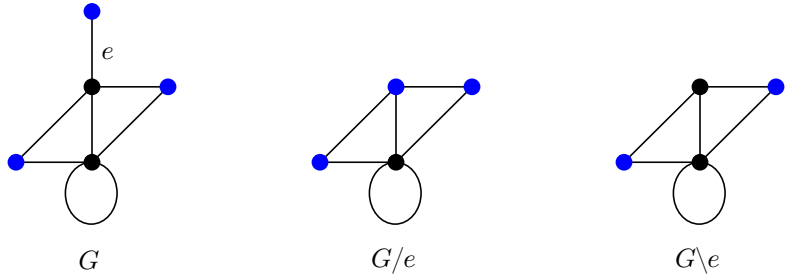
\begin{figure}[ht]
\centering
\begin{tikzpicture}
\draw[draw=black, semithick, solid] (-5.00,-1.00) -- (-4.00,-1.00);
\draw[draw=black, semithick, solid] (-4.00,-1.00) -- (-4.00,0.00);
\draw[draw=black, semithick, solid] (-5.00,-1.00) -- (-4.00,0.00);
\draw[draw=black, semithick, solid] (-4.00,-1.00) -- (-3.00,0.00);
\draw[draw=black, semithick, solid] (-4.00,0.00) -- (-3.00,0.00);
\draw[draw=black, semithick, solid] (-4.00,0.00) -- (-4.00,1.00);
\draw[draw=black, semithick, solid] (-1.00,-1.00) -- (0.00,-1.00);
\draw[draw=black, semithick, solid] (0.00,-1.00) -- (0.00,0.00);
\draw[draw=black, semithick, solid] (0.00,0.00) -- (1.00,0.00);
\draw[draw=black, semithick, solid] (0.00,-1.00) -- (1.00,0.00);
\draw[draw=black, semithick, solid] (0.00,0.00) -- (-1.00,-1.00);
\draw[draw=black, semithick, solid] (3.00,-1.00) -- (4.00,-1.00);
\draw[draw=black, semithick, solid] (4.00,-1.00) -- (5.00,0.00);
\draw[draw=black, semithick, solid] (5.00,0.00) -- (4.00,0.00);
\draw[draw=black, semithick, solid] (4.00,0.00) -- (4.00,-1.00);
\draw[draw=black, semithick, solid] (3.00,-1.00) -- (4.00,0.00);
\draw[draw=black, semithick, solid] (-4.01,-1.41) ellipse (0.34 and -0.41);
\draw[draw=black, semithick, solid] (-0.01,-1.41) ellipse (0.34 and -0.41);
\draw[draw=black, semithick, solid] (3.99,-1.41) ellipse (0.34 and -0.41);

\draw[draw=blue, fill=blue, semithick, solid] (-5.00,-1.00) circle (0.1);
\draw[draw=black, fill=black, semithick, solid] (-4.00,-1.00) circle (0.1);
\draw[draw=black, fill=black, semithick, solid] (-4.00,0.00) circle (0.1);
\draw[draw=blue, fill=blue, semithick, solid] (-4.00,1.00) circle (0.1);
\draw[draw=blue, fill=blue, semithick, solid] (-3.00,0.00) circle (0.1);

\draw[draw=blue, fill=blue, semithick, solid] (-1.00,-1.00) circle (0.1);
\draw[draw=black, fill=black, semithick, solid] (0.00,-1.00) circle (0.1);
\draw[draw=blue, fill=blue, semithick, solid] (0.00,0.00) circle (0.1);
\draw[draw=blue, fill=blue, semithick, solid] (1.00,0.00) circle (0.1);

\draw[draw=blue, fill=blue, semithick, solid] (3.00,-1.00) circle (0.1);
\draw[draw=black, fill=black, semithick, solid] (4.00,-1.00) circle (0.1);
\draw[draw=black, fill=black, semithick, solid] (4.00,0.00) circle (0.1);
\draw[draw=blue, fill=blue, semithick, solid] (5.00,0.00) circle (0.1);

\node[black, anchor=north west] at (-4.3,-2.045) {$G$};
\node[black, anchor=south west] at (-4,0.25) {$e$};
\node[black, anchor=north west] at (-0.45,-2) {$G/e$};
\node[black, anchor=north west] at (3.55,-2) {$G\sm e$};
\end{tikzpicture}

\caption{A graph $(G,l)$ with labelled vertices in blue (left). It has a leaf $e$ whose valence-one vertex is labelled. The graphs $G/ e$ (middle), and  $G\sm e$ (right) are homeomorphic but do not have the same induced labelling.}
\label{fig:collapseVsQuotient}
\end{figure}

\subsection{Posets of core graphs}

Let $\Core(G,l)$ be the poset of proper core subgraphs of a labelled graph $(G,l)$, ordered by inclusion.
Let $\NonTrees(G,l)$ be the poset of proper subgraphs of $G$ such that at least one connected component has non-trivial fundamental group or contains at least two labelled vertices.
The poset $\Core(G,l)$ is contained in $\NonTrees(G,l)$, and indeed $\NonTrees(G,l)$ is the upward-closure of $\Core(G,l)$ in the poset of proper subgraphs of $G$.

\begin{lemma}
\label{lem_X_and_Core}
Let $(G,l)$ be a labelled graph.
Then $\NonTrees(G,l)$ deformation retracts to $\Core(G,l)$.
\end{lemma}
\begin{proof}
This is because every $H\in \NonTrees(G,l)$ contains a unique maximal core subgraph $\MaxCore{H}\in \Core(G,l)$.
The assignment $H\mapsto \MaxCore{H}$ defines a monotone poset map $\NonTrees(G,l) \to \Core(G,l)$ such that $H \geq \MaxCore{H}$
 and restricts to the identity on $\Core(G,l)$. Hence $\NonTrees(G,l)$ deformation retracts to $\Core(G,l)$ (see \cite[Corollary 2.5]{Bru:buildingsfreefactora}).
\end{proof}

In \cite[Proposition 4.14]{Bru:buildingsfreefactora}, it was shown that for $k(l) = 0$, the poset $\NonTrees(G,l)$ is contractible if and only if $G$ has a separating edge.
The following definition is set up in such a way that this result generalises to the cases $k(l)>0$.

\begin{definition}
\label{def_l_separating}
We say that an edge $e$ of $(G,l)$ is \textit{$l$-separating} if it satisfies the following three conditions:
\begin{enumerate}
\item The edge $e$ is separating.
\item \label{it_l_sep_all_on_one_side} One of the components of $G\sm e^\circ$ contains all the labelled vertices.
\item \label{it_l_sep_in_X} The graph $G\sm e$ has at least one connected component that has non-trivial fundamental group or contains at least two labelled vertices, i.e.~$G\sm e \in \NonTrees(G,l)$.
\end{enumerate}
\end{definition}

\begin{remark}
\label{rem:l_separating_k_leq_1}
For $k\leq 1$, an edge is $l$-separating if and only if it is separating and $G$ has non-trivial fundamental group.
In particular, the $l$-separating edges of $(G,l)$ are the same as the $\emptyset$-separating edges of the trivially labelled graph $(G,\emptyset)$.
\end{remark}

\begin{lemma}
\label{lem_con_X}
Let $(G,l)$ be a connected labelled graph with $k = k(l)$ labelled points and $\rk(G)=n$. Let
\begin{equation*}
d(G,l) = 
\begin{cases}
n+k-3 & \text{if }k\geq 1,\\
n-2 & \text{if }k = 0.
\end{cases}
\end{equation*}
Then $\NonTrees(G,l)$ is homotopy equivalent to a wedge of $d(G,l)$-spheres. It is contractible if and only if there is an $l$-separating edge.
\end{lemma}

To explain the case distinction, note that if $k(l) = 1$, then $\NonTrees(G,l) = \NonTrees(G,\emptyset)$.
In \cite[Proposition 4.14]{Bru:buildingsfreefactora}, \cref{lem_con_X} was deduced from \cite[Proposition 2.2]{Vog:Localstructuresome} for the case $k=0$, $n\geq 2$ under the assumption that every vertex of $G$ has valence at least three.

\subsection*{Proof of \texorpdfstring{\cref{lem_con_X}}{Lemma 2.3}}
Suppose first that $G$ contains an $l$-separating edge $e$. We want to show that $\NonTrees(G,l)$ is contractible.
Let $H\in \NonTrees(G,l)$, i.e.~it has at least one connected component that has trivial fundamental group or contains at least two labelled vertices. We claim that the same is true for $H\sm \ls e \rs$. If $e\not \in E(H)$, this is clear. So assume $e \in E(H)$. Note that $e\not\in \NonTrees(G,l)$, so $H\sm \ls e \rs\neq \emptyset$. If $H$ has a connected component with non-trivial fundamental group, the same is true for $H\sm \ls e \rs$ as $e$ is separating in $H$. On the other hand, if $H$ has a component with at least two labelled vertices, then by \cref{it_l_sep_all_on_one_side} of \cref{def_l_separating}, the same is true for $H\sm \ls e \rs$. Hence $H\sm \ls e \rs$ is an element of $\NonTrees(G,l)$, as claimed.
Using this, the assignment $H\mapsto H\sm \ls e \rs$ gives a monotone poset map $\NonTrees(G,l) \to \NonTrees(G,l)$, so a homotopy equivalence to its image. Using \cref{it_l_sep_in_X} of \cref{def_l_separating}, $G \sm e $ is an element of $\NonTrees(G,l)$, so it lies in this image. It clearly forms a maximal element, so the image and hence $\NonTrees(G,l)$ are contractible. 

Now assume that $G$ contains no $l$-separating edges.
We prove that $\NonTrees(G,l)$ is a non-trivial wedge of $d(G,l)$-spheres by induction on the number of edges of $G$. If $G$ has just a single edge, then this edge is not $l$-separating and $\NonTrees(G,l)$ is empty, so in particular non-contractible (this is true independently of whether the edge forms a loop and whether it has labelled endpoints or not). Hence, the claim is true.
In what follows, we will assume that there are at least two edges.

\subsubsection*{Removing loops}
Suppose that $G$ contains a loop $e$.
We claim that $G \sm e$ is an element of $\NonTrees(G,l)$.
As $G$ has at least two edges, $G \sm e$ is non-empty.
If $G \sm e$ was a tree with only one labelled vertex, then also $G$ has just one labelled vertex and all edges except $e$ are separating. But by \cref{rem:l_separating_k_leq_1}, they are $l$-separating, which contradicts our assumption.
Hence, $G \sm e$ either has non-trivial fundamental group or at least two labelled vertices, which proves that $G \sm e\in \NonTrees(G,l)$.

Using this, the poset $\NonTrees(G,l)$ can be written as the union of two subposets $X_1$ and $X_2$ as follows.
Let
\begin{equation*}
	X_1 = \NonTrees(G\sm e,l')\cup \ls G \sm e \rs,
\end{equation*}
where $l'$ is the labelling induced by $l$.
This subposet has $G \sm e$ as its unique maximal element and hence is contractible. The second subposet is given by $\NonTrees(G\sm e,l')$ together with all proper subgraphs of $G$ containing $e$ (such a subgraph is always contained in $\NonTrees(G,l)$ because it contains the loop $e$ and hence has non-trivial fundamental group). That is,
\begin{equation*}
	X_2 = \NonTrees(G \sm e,l') \cup \ls H\subsetneq G \mid e\in H \rs.
\end{equation*}
This poset is contractible as well: The assignment $H\mapsto H\cup\ls e \rs$ gives a monotone poset map $X_2\to X_2$ and its image is contractible because it has a unique minimal element $\ls e\rs$.
Both $X_1$ and $X_2$ are downward closed, so 
\begin{equation*}
	|\NonTrees(G,l)| = |X_1|\cup_{|X_1\cap X_2|} |X_2|, \text{ where } X_1 \cap X_2 = \NonTrees(G\sm e,l').
\end{equation*}
Thus $\NonTrees(G,l)$ is homotopy equivalent to the topological suspension of $\NonTrees(G\sm e, l')$.
As $G$ has no $l$-separating edge, also $G\sm e$ has no $l'$-separating edge. The graph $G\sm e$ is connected, so by induction, $\NonTrees(G\sm e,l')$ is homotopy equivalent to a non-trivial wedge of $d(G\sm e,l')$-spheres. Furthermore, $G\sm e$ has rank one less than $G$ but the same number of labelled points, i.e.~$\rk(G\sm e) = n-1$ and $k(l')=k$. Hence, $d(G\sm e,l')=d(G,l)-1$. This implies the claim for $\NonTrees(G,l)$.

\subsubsection*{Loop-free graphs}
By the previous paragraph, we can now assume that $G$ contains no loops.
If $G$ is a tree, let $e$ be a leaf of $G$. Otherwise, let $e$ be a non-separating (and non-loop) edge. Note that in either case, $G\sm e$ is connected. We will prove the two cases in parallel.
Let $Y_1:= \NonTrees(G,l) \sm \ls G \sm e \rs$ and denote by $(G\sm e,l')$ the labelled graph obtained from $G$ by removing $e$.

\begin{claim}
\label[claim]{claim_Y_1}
The poset $Y_1$ is homotopy equivalent to a wedge of $d(G,l)$-spheres. If $Y_1$ is contractible, then $Y_1 \neq \NonTrees(G,l)$ and $G\sm e$ has no $l'$-separating edge.
\end{claim}
\begin{proof}
To prove this, we consider two cases.
The first is that $e$ has both endpoints labelled. We claim that in this case, $Y_1$ is contractible and different from $\NonTrees(G,l)$.
As $G$ has more than one edge and both endpoints of $e$ are labelled, we have $\ls e \rs \in Y_1$. We get a monotone map $Y_1 \to Y_1$ by the assignment $H\mapsto H\cup \ls e\rs$. Its image has a unique minimal element $\ls e\rs$, so $Y_1$ is contractible.
It remains to show that $Y_1$ is different from $\NonTrees(G,l)$, i.e.~that $G \sm e \in \NonTrees(G,l)$.
Recall that $G \sm e$ is connected and that either $G$ is a tree with $e$ a leaf or $e$ is non-separating.
If $e$ is non-separating, then $G \sm e$ still contains both (labelled) endpoints of $e$. Hence, we have indeed  $G \sm e \in\NonTrees(G,l)$. If, on the other hand, $G$ is a tree and $e$ is a leaf, the assumption that $G$ has no $l$-separating edge implies that there must be at least one labelled vertex that is not an endpoint of $e$. Hence, $G \sm e$ is a connected graph with at least two labelled vertices and we again have $G \sm e \in\NonTrees(G,l)$.

The second case is that at least one of the endpoints of $e$ is unlabelled. 
In this case, for all $H\in Y_1$, the image of $H$ in $G/e$ is also an element of $\NonTrees(G/e,\tilde{l})$, where $\tilde{l}$ is the labelling induced by $l$. (This is because collapsing $e$ does not change the fundamental group as $e$ is not a loop and also does not change the number of labelled vertices as at most one endpoint of $e$ is unlabelled.)
Hence, we get a poset map
\begin{align*}
f\colon Y_1 &\to \NonTrees(G/e,\tilde{l})
\end{align*}
that sends a subgraph to its image in the quotient $G/e$.
We define a poset map in the other direction as follows. Let $v_e$ be the vertex of $G/e$ to which $e$ was collapsed. Then define
\begin{align*}
g\colon\NonTrees(G/e,\tilde{l}) &\to Y_1 \\
K &\mapsto \begin{cases}
K \cup \ls e \rs & v_e\in V(K), \\
K & \text{otherwise}.
\end{cases} 
\end{align*}
Here we have identified $K$ with the set of corresponding edges in $G$.
We have $g\circ f (H) \supseteq H$ and $f\circ g (K)= K$. Hence, $Y_1$ and $\NonTrees(G/e, \tilde{l})$ have the same homotopy type (see \cref{sec_poset_notation}).
The graph $G/e$ has no $\tilde{l}$-separating edges (since $(G,l)$ has no $l$-separating edge), so by induction, we see that $Y_1 \simeq \NonTrees(G/e,\tilde{l})$ is homotopy equivalent to a non-trivial wedge of $d(G/e,\tilde{l})$-spheres.
Since $e$ is not a loop, we have $\rk(G/e) = \rk(G)$. We also have $k(\tilde{l}) = k(l)$ because not both endpoints of $e$ are labelled.
Hence, $d(G/e,\tilde{l}) = d(G,l)$.
\end{proof}

Now we use this subposet to show the desired connectivity for $\NonTrees(G,l)$.
If $Y_1  = \NonTrees(G,l)$ then the result immediately follows from Claim \ref{claim_Y_1}.
Thus we may assume that $G\sm e\in \NonTrees(G,l)$.
It follows that $\NonTrees(G,l)$ is obtained from $Y_1 = \NonTrees(G,l) \sm \ls G \sm e \rs$ by attaching the star of $G\sm e$ along its link $\lk(G\sm e)$.
That is, if $Y_2$ is the downward closure of $G\sm e$ in $\NonTrees(G,l)$ (i.e. the lower interval $\NonTrees(G,l)_{\subseteq G\sm e}$), then
\begin{equation*}
	|\NonTrees(G,l)| = |Y_1|\cup_{|Y_1\cap Y_2|} |Y_2|, \text{ where } Y_1\cap Y_2 = \NonTrees(G,l)_{\subsetneq G\sm e}.
\end{equation*}
Here, $\lk(G\sm e)$ is the order complex of $\NonTrees(G,l)_{\subsetneq G\sm e}$.
The poset $Y_2$ is contractible as it has the unique maximal element $G\sm e$, and $Y_1$ is homotopy equivalent to a wedge of $d(G,l)$-spheres by Claim \ref{claim_Y_1}.
Hence, the following claim implies that $\NonTrees(G,l)$ is homotopy equivalent to a non-trivial wedge of $d(G,l)$-spheres and finishes the proof of \cref{lem_con_X}.
\begin{claim}
The poset $\NonTrees(G,l)_{\subsetneq G\sm e}$ is homotopy equivalent to a wedge of $(d(G,l)-1)$-spheres and either $\NonTrees(G,l)_{\subsetneq G\sm e}$ or $Y_1$ are non-contractible.
\end{claim}
\begin{proof}
Note that $\NonTrees(G,l)_{\subsetneq G\sm e}$ is isomorphic to $\NonTrees(G\sm e,l')$, where $l'$ is the labelling induced by $l$.
Since $G\sm e$ is connected, by induction, $\NonTrees(G\sm e,l')$ is homotopy equivalent to a wedge of $d(G\sm e,l')$-spheres, and it is not contractible if $G\sm e$ contains no $l'$-separating edge. By Claim \ref{claim_Y_1}, either this is the case or $Y_1$ is non-contractible. It remains to show that either $d(G\sm e,l')$ and $d(G,l)-1$ are equal or both are smaller than $0$.

Here, we need to distinguish the cases where $e$ is either non-separating or a leaf. If $e$ is non-separating, then $\rk(G\sm e) = \rk(G)-1$ and $k(l')=k$. Hence the equality $d(G\sm e,l') = d(G,l)-1$ holds.
Now assume that $e$ is a leaf. If $k\leq 1$ and $n=0$, then $\NonTrees(G\sm e,l') = \emptyset$. This is consistent with $d(G,l) = -2$.
Otherwise, $(G,l)$ not having $l$-separating edges implies that the valence-one vertex of $e$ is labelled and that $k\geq 2$. Hence, we have $k(l') = k-1\geq 1$ and $\rk(G\sm e) = \rk(G)$, so we again get equality $d(G\sm e,l') = d(G,l)-1$.
\end{proof}

\section{Proof of the main theorems}
\label{sec_proof_frames}

In this section, we prove \cref{thm_connectivity_frames}, \cref{thm_CM_frames_intro} and \cref{thm_connectivity_free_factor_systems}.
We will frequently use Quillen's fibre theorem as stated in \cite[Theorem 5.3]{PW} to show that certain maps are highly connected
(that is, they induce isomorphism between the homotopy groups up to a large enough degree).

Throughout the section, we fix a finitely generated group $A$ and a (possibly empty) free factor system $\mcA$ of $A$
, and write $\FFS = \FFS(A,\mcA)$.
We set
\begin{equation*}
	n = \corank(\mcA) \text{ and } k =  |\mcA|.
\end{equation*}
We denote by $\FreeS = \FreeS(A,\mcA)$ the poset of free splittings of $A$ relative to $\mcA$.
Its elements can be seen as pairs $S = (\mathbb{G}, \mathfrak{m})$, where $\mathbb{G}$ is a graph of groups with underlying labelled graph a connected core graph in the sense of \cref{sec_posets_of_graphs}, and where $\mathfrak{m}\colon A\to \pi_1(\mathbb{G})$ is an isomorphism which is well-defined up to inner automorphisms.
The poset relation is given by saying that $S>S'$ if $S'$ is obtained from $S$ by collapsing a subgraph.
For $S\in \FreeS$, we denote by $\mcV(S)$ the free factor system given by the conjugacy classes of its vertex groups. By definition, we have $\mcV(S)\in\FFS \cup \{ \mcA\}$, i.e.~$\mcA\sqsubseteq \mcV(S)$.
We refer the reader to \cite[Section 4.1]{Bru:buildingsfreefactora} for more detailed definitions.

We set $ L = L(A,\mcA) = \{ S \in \FFS \tq \mcV(S) = \mcA \}$ and $\bFreeS = \FreeS \sm L$.
Here, $L$ is the spine of relative Outer space $\mathcal{O}(A,\mcA)$ in the sense of \cite{GL:outerspacefree}, and $\bFreeS$ is its simplicial boundary.
This simplicial boundary has the homotopy type of the poset $\FFS$ that we are interested in:

\begin{lemma}
\label{lem_bFreeS_FFS}
We have a homotopy equivalence $\mcV :\bFreeS \to \FFS^{\op}$.
\end{lemma}
\begin{proof}
This is \cite[Proposition 4.30.1]{Bru:buildingsfreefactora}.
\end{proof}

Let $Z\subseteq L\times \bFreeS$ be the subposet consisting of all pairs $(S,S')$ where $S'$ is obtained from $S$ by collapsing a proper core subgraph. We have projections $p_1\colon Z \to L$ and $p_2\colon Z \to \bFreeS$.

\begin{lemma}
\label{lem_p2_equivalence}
The projection $p_2:Z\to \bFreeS$ is a homotopy equivalence.
\end{lemma}  
\begin{proof}
This is part of \cite[Corollary 4.35]{Bru:buildingsfreefactora}.
\end{proof}

Now we prove that $p_1$ is highly connected.
Section 4.6.4 of \cite{Bru:buildingsfreefactora} studies this map for the case $A = F_n$ and $\mcA = \emptyset$.
Here we use \cref{lem_con_X} to generalise this to an arbitrary pair $(A,\mcA)$.

\begin{lemma}
\label{lem_p1_equivalence}
The projection $p_1 : Z\to L$ is an $e(n,k)$-equivalence, where
\begin{equation*}
e(n,k) = 
\begin{cases}
n+k-3 & \text{if }k\geq 2,\\
n-1 & \text{if }k = 1,\\
n-2 & \text{if }k = 0.
\end{cases}
\end{equation*}
\end{lemma}
Note that $e(n,k)$ agrees with $d(G,l)$ from \cref{lem_con_X} except if $k=1$. This case distinction will become clear in the following proof.
\begin{proof}[Proof of \cref{lem_p1_equivalence}]
Let $S\in L$ and let $(G,l)$ be the underlying labelled graph.
Note that $G$ has exactly $k$ labelled vertices and $\rk(G)=n$.
By \cite[Lemma 4.36]{Bru:buildingsfreefactora}, the fibre $p_1^{-1}(L_{\leq S})$ is homotopy equivalent to $p_1^{-1}(S)$. (The proof of \cite[Lemma 4.36]{Bru:buildingsfreefactora} is formulated for $A = F_n$, $\mcA = \emptyset$, but it generalises verbatim to the setup here.)
The latter is exactly the poset $\Core(G,l)$ defined in \cref{sec_posets_of_graphs}. This is $(d(G,l)-1)$-connected by \cref{lem_X_and_Core} and \cref{lem_con_X}. For $k\neq 1$, we have $d(G,l) = e(n,k)$ and the claim directly follows from Quillen's fibre theorem \cite[Theorem 5.3]{PW}.

For $k=1$, we have $d(G,l) = n-2$, so the same argument would give us one degree of connectivity less than claimed. 
Assume we are in this case. We will prove that
\[ p_1^{-1}(L_{\leq S}) * L_{>S} \]
is $(n-2)$-connected for all $S\in L$. Applying Quillen's fibre theorem then gives us the slightly stronger connectivity statement.
Note that $n\neq 0$ since $\mcA \neq \{ [A]\}$. If $L_{>S}$ is non-empty (i.e.~$S$ is not maximal), then $p_1^{-1}(L_{\leq S}) * L_{>S}$ is at least $(n -2)$-connected.
Suppose now that $S$ is maximal. We prove that this implies that $G$ contains an $l$-separating edge. Since $G$ has only one labelled vertex, say $v$, and $\rk(G) = n\neq 0$, it is enough to show that $v$ has valence one (and hence the $l$-separating edge is the leaf containing $v$).
Indeed, if $v$ has valence at least two, then we can ``push away'' the labelling from $v$ by adding a new edge with one endpoint $v$ and the other endpoint containing the labelling that $v$ had in $G$ (and thus $v$ is now unlabelled).
See \cref{fig:expandingNonMaximal}.
This graph defines an element of $L$ that lies in $L_{>S}$ (all unlabelled vertices are at least trivalent), contradicting the maximality of $S$.
Therefore $v$ must be a vertex of valence one.
Thus $G$ contains an $l$-separating edge, and by 
\cref{lem_con_X}, $p_1^{-1}(L_{\leq S}) \simeq \Core(G,l) \simeq \NonTrees(G,l)$ is contractible. In particular this implies that $p_1^{-1}(L_{\leq S}) * L_{>S}$ is also $(n-2)$-connected.
This concludes the proof.
\end{proof}

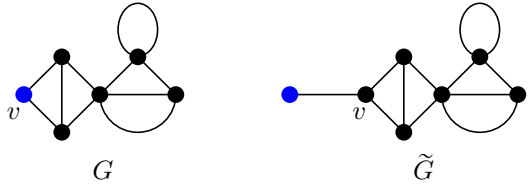
\begin{figure}[t]
\centering
\begin{tikzpicture}
\draw[draw=black, semithick, solid] (-3.00,0.00) -- (-2.50,0.50);
\draw[draw=black, semithick, solid] (-3.00,0.00) -- (-2.50,-0.50);
\draw[draw=black, semithick, solid] (-2.50,-0.50) -- (-2.50,0.50);
\draw[draw=black, semithick, solid] (-2.50,-0.50) -- (-2.00,0.00);
\draw[draw=black, semithick, solid] (-2.50,0.50) -- (-2.00,0.00);
\draw[draw=black, semithick, solid] (-2.00,0.00) -- (-1.00,0.00);
\draw[draw=black, semithick, solid] (-2.00,0.00) -- (-1.50,0.50);
\draw[draw=black, semithick, solid] (-1.50,0.50) -- (-1.00,0.00);
\draw[draw=black, semithick, solid] (-1.50,0.86) ellipse (0.26 and 0.35);
\draw[draw=black, semithick, solid] ([shift=(180:0.50 and -0.50)]-1.50,0.00) arc (180:0:0.50 and -0.50);
\draw[draw=blue, fill=blue, semithick, solid] (-3.00,0.00) circle (0.1);
\draw[draw=black, fill=black, semithick, solid] (-2.00,0.00) circle (0.1);
\draw[draw=black, fill=black, semithick, solid] (-1.50,0.50) circle (0.1);
\draw[draw=black, fill=black, semithick, solid] (-2.50,0.50) circle (0.1);
\draw[draw=black, fill=black, semithick, solid] (-2.50,-0.50) circle (0.1);
\draw[draw=black, fill=black, semithick, solid] (-1.00,0.00) circle (0.1);
\node[black, anchor=south west] at (-2.22,-1.25) {$G$};
\node[black, anchor=south west] at (-3.35,-0.45) {$v$};

\draw[draw=black, semithick, solid] (1.50,0.00) -- (2.00,0.50);
\draw[draw=black, semithick, solid] (1.50,0.00) -- (2.00,-0.50);
\draw[draw=black, semithick, solid] (2.00,-0.50) -- (2.00,0.50);
\draw[draw=black, semithick, solid] (2.00,-0.50) -- (2.50,0.00);
\draw[draw=black, semithick, solid] (2.50,0.00) -- (2.00,0.50);
\draw[draw=black, semithick, solid] (2.50,0.00) -- (3.50,0.00);
\draw[draw=black, semithick, solid] (2.50,0.00) -- (3.00,0.50);
\draw[draw=black, semithick, solid] (3.00,0.50) -- (3.50,0.00);
\draw[draw=black, semithick, solid] (1.50,0.00) -- (0.50,0.00);
\draw[draw=black, semithick, solid] (3.00,0.86) ellipse (0.26 and 0.35);
\draw[draw=black, semithick, solid] ([shift=(180:0.50 and -0.50)]3.00,0.00) arc (180:0:0.50 and -0.50);
\draw[draw=black, fill=black, semithick, solid] (2.00,-0.50) circle (0.1);
\draw[draw=black, fill=black, semithick, solid] (2.00,0.50) circle (0.1);
\draw[draw=black, fill=black, semithick, solid] (1.50,0.00) circle (0.1);
\draw[draw=black, fill=black, semithick, solid] (3.50,0.00) circle (0.1);
\draw[draw=blue, fill=blue, semithick, solid] (0.50,0.00) circle (0.1);
\draw[draw=black, fill=black, semithick, solid] (3.00,0.50) circle (0.1);
\draw[draw=black, fill=black, semithick, solid] (2.50,0.00) circle (0.1);
\node[black, anchor=south west] at (2,-1.25) {$\widetilde{G}$};
\node[black, anchor=south west] at (1.2,-0.45) {$v$};
\end{tikzpicture}

\caption{Left: A graph $G$ with a labelled vertex $v$ (in blue) of valence at least two. Right: The graph $\widetilde{G}$ obtained by pushing away the labelling from $v$ and creating a leaf. We have $G < \widetilde{G}$ and every non-labelled vertex of $\widetilde{G}$ has valence at least three.}
\label{fig:expandingNonMaximal}
\end{figure}

As a corollary, we get \cref{thm_connectivity_free_factor_systems}, which states that $\FFS$ is $(n+k-4)$-connected for $k\geq 2$ and $(n-2)$-connected for $k\in \ls 0,1\rs$:

\begin{proof}[Proof of \cref{thm_connectivity_free_factor_systems}]
As pointed out in the introduction, the case $k=0$, i.e.~$\mcA = \emptyset$, only occurs if $A$ is a free group. In this case, $n = \rk(A)$ and \cref{thm_connectivity_free_factor_systems} follows from \cite[Theorem B]{BG:Homotopytypecomplex}.
Hence, we can assume that $k \neq 0$; we can also assume that $n+k\geq 2$ as there is nothing to show otherwise. Under these assumptions, the poset $L$ is contractible by \cite[Section 4]{GL:outerspacefree}.
Thus by \cref{lem_bFreeS_FFS}, \cref{lem_p2_equivalence} and \cref{lem_p1_equivalence}, $\FFS$ is $(n + k-4)$-connected for $k\geq 2$, and $(n-2)$-connected if $k=1$.
In any case, $\FFS$ is always at least $(n-2)$-connected.
\end{proof}

The following is a corollary of \cite[Lemma 2.5]{HandelMosher}.

\begin{proposition}
\label{prop_lower_intervals}
Let $\mcA' \in \FFS = \FFS(A,\mcA)$. Then
\[  \FFS_{\sqsubseteq \mcA'}  \cong \left( \  \prod_{[B]\in \mcA'} \left( \FFS(B, \mcA_{[B]})\cup \{ \, \{[B]\},\emptyset \, \} \right) \ \right)  \sm \{ \, (\emptyset,\ldots,\emptyset) \, \} ,\]
where $\mcA_{[B]} = \{ \, [C]\in \mcA \tq \{[C]\}\sqsubseteq \{[B]\}\, \}$.
\end{proposition}

If $\{[B]\} \in \FFS(A,\mcA)\cup \{[A]\}$, we write $\corank_B(\mcA)$ for the corank of $\mcA$ viewed as a free factor system of $B$.
For $\mcA'\in \FFS$, we have
\begin{equation}
\label{eq_rank_and_corank}
\corank_A(\mcA) = \corank_A(\mcA') + \sum_{[B]\in \mcA'} \corank_B(\mcA_{[B]}).
\end{equation}

We are now ready to prove \cref{thm_CM_frames_intro}, i.e.~to show that $\Frames := \Frames(A,\mcA)$ is Cohen--Macaulay of dimension $n-1$.

\begin{proof}[Proof of \cref{thm_CM_frames_intro}]
We argue by induction on $n+k$.
If $k=0$, i.e.~$\mcA = \emptyset$, the base case is $n = 1$, and here by definition $\Frames = \{ [A] \}$, which is contractible.
Hence we can assume that $n \geq 2$ if $\mcA = \emptyset$.
Furthermore, if $n = 0$ then $\Frames$ is empty, so we can also assume that $n \geq 1$ if $k\neq 0$.
In particular, we always have a poset inclusion $i:\Frames \hookrightarrow \FFS$.

In view of the previous paragraph, it is enough to show that the inclusion map $i$ is an $(n-1)$-equivalence, and that the codomain is $(n-2)$-connected.
To that end, we investigate the fibres and apply Quillen's fibre theorem.

Let $\mcA' \in \FFS\sm \Frames$.
By \cref{prop_lower_intervals}, we have
\[ i^{-1}\big( \ \FFS_{\sqsubseteq \mcA'} \ \big) \cong  \left( \ \prod_{[B]\in \mcA'} (\Frames(B,\mcA_{[B]})\cup \{ \emptyset\}) \ \right) \sm \{ (\emptyset,\ldots,\emptyset)\} . \]
This poset is homotopy equivalent to the join
\begin{equation}
\label{eq_join_frames}
\bigast_{[B]\in \mcA'} \Frames(B,\mcA_{[B]}).
\end{equation}
Since a class $[B]\in \mcA'$ gives a proper subgroup $B < A$, we have 
\begin{equation*}
	\corank_{B}(\mcA_{[B]})+|\mcA_{[B]}| < \corank_A(\mcA)+|\mcA| = n+k
\end{equation*}
by \cref{eq_rank_and_corank}.
Therefore, by induction, each $ \Frames(B,\mcA_{[B]}) $ is $(\corank_{B}(\mcA_{[B]}) - 2)$-connected, and so the join of \cref{eq_join_frames} is
\[ (2|\mcA'| - 2) + \sum_{[B]\in\mcA'} (\corank_{B}(\mcA_{[B]})-2) = n - \corank_A(\mcA') - 2\]
connected.

On the other hand, since $\mcA \sqsubset \mcA'$, by \cref{thm_connectivity_free_factor_systems} we have that
\[ \FFS_{\sqsupset \mcA'} = \FFS(A,\mcA') \]
is at least $(\corank(\mcA')-2)$-connected.
Hence the join 
\[ i^{-1}\big( \ \FFS_{\sqsubseteq \mcA'} \ \big) * \FFS_{\sqsupset \mcA'}\]
is at least $(n-2)$-connected.
By Quillen's fibre theorem, the map $i$ is an $(n-1)$-equivalence.

Finally, the codomain $\FFS$ is at least $(n-2)$-connected by \cref{thm_connectivity_free_factor_systems}.
Therefore $\Frames$ is $(n-2)$-connected.

To conclude with the proof, we show the Cohen--Macaulayness.
As noted in the paragraph before \cref{thm_CM_frames_intro}, we can regard $\Frames$ as a simplicial complex. The link of a simplex $\mcA'\in  \Frames \cup \{\emptyset\}$ can be identified with $\Frames(A,\mcA')$, which we saw is $(\corank(\mcA')-2)$-connected.
But this link has dimension $\corank(\mcA')-1$, so it is spherical of this dimension.
Thus $\Frames$ is Cohen--Macaulay of dimension $n-1$.
\end{proof}

Before we conclude with the proof of \cref{thm_connectivity_frames}, we recall the notion of inflation of a simplicial complex.
The following definition is a generalisation of the one given in \cite{BWW:PosetFiberTheorems}, in the context of inflation of finite simplicial complexes.
The adaptation presented here corresponds to Definition 4.3 in \cite{PW}.

\begin{definition}
\label{def:inflation}
Let $K$ be any simplicial complex (not necessarily finite) and $\mcP = (P_v)_{v\in V(K)}$  a family of non-empty sets indexed on the vertex set $V(K)$ of $K$. The \textit{inflation of $K$ by $\mcP$}, which we denote by $(K,\mcP)$, is the simplicial complex whose simplices are of the form $\{ (v_0,p_0),\ldots, (v_k,p_k)\}$ such that $\{v_0,\ldots,v_k\}$ is a $k$-simplex of $K$ and $p_i \in P_{v_i}$ for all $i$.
\end{definition}

There is a natural map from the inflation to the original complex $p:(K,\mcP)\to K$ that takes a simplex $\{ (v_0,p_0),\ldots, (v_k,p_k)\} \in (K,\mcP)$ to $\{v_0,\ldots,v_k\}\in K$.
Note that $(K, \mcP)$ and $K$ have the same dimension.
By studying the fibres of this map, one can conclude that, if $K$ is finite-dimensional, then $(K,\mcP)$ is Cohen--Macaulay if and only if $K$ is.
See Proposition 5.10 of \cite{PW} for more details.

\begin{proof}[Proof of \cref{thm_connectivity_frames}]
Let $A$ be a free group of finite rank $n$.
The complex $\PartialBases(A)$ is the inflation of $\Frames(A,\emptyset)$
in the sense of \cref{def:inflation}.
Precisely, for a vertex $\{[H]\}$ in $\Frames(A,\emptyset)$, we take $P_{[H]}$ as the set of $A$-conjugacy classes of generators of $H$.
Since $\Frames(A,\emptyset)$ is Cohen--Macaulay of dimension $n-1$ by \cref{thm_CM_frames_intro}, $\PartialBases(A)$ is also Cohen--Macaulay of dimension $n-1$ by \cite[Proposition 5.10]{PW}.
\end{proof}

\printbibliography

\end{document}